\documentclass[12pt]{article}
\usepackage{amsfonts}
\usepackage[pctex32]{graphics}
\usepackage[dvips]{graphicx}
\usepackage{amsthm}
\usepackage{amsmath}
\usepackage{amssymb}
\usepackage[all]{xy}
\usepackage{latexsym}
\usepackage{layout}
\usepackage{epsfig}
\usepackage{graphicx}
\usepackage{pstricks,pst-node,pst-plot} %This loads the pstricks package
\usepackage{color} %This loads the color package
\title{Maximal Symmetric Difference-Free Families of Subsets of $[n]$}
\author {Travis G.~Buck and Anant P.~Godbole\\
Department of Mathematics and Statistics\\
East Tennessee State University}
\begin{document}
\def\qed{\vbox{\hrule\hbox{\vrule\kern3pt\vbox{\kern6pt}\kern3pt\vrule}\hrule}}
\def\ms{\medskip}
\def\n{\noindent}
\def\ep{\varepsilon}
\def\G{\Gamma}
\def\lr{\left(}
\def\lf{\lfloor}
\def\rf{\rfloor}
\def\lc{\left\{}
\def\rc{\right\}}
\def\rr{\right)}
\def\ph{\varphi}
\def\p{\mathbb P}
\def\v{\mathbb V}
\def\nk{n \choose k}
\def\a{\cal A}
\def\e{\mathbb E}
\def\l{\mathbb L}
\newcommand{\bsno}{\bigskip\noindent}
\newcommand{\msno}{\medskip\noindent}
\newcommand{\oM}{M}
\newcommand{\omni}{\omega(k,a)}
\newtheorem{thm}{Theorem}[section]
\newtheorem{con}{Conjecture}[section]
\newtheorem{deff}[thm]{Definition}
\newtheorem{lem}[thm]{Lemma}
\newtheorem{cor}[thm]{Corollary}
\newtheorem{rem}[thm]{Remark}
\newtheorem{prp}[thm]{Proposition}
\newtheorem{ex}[thm]{Example}
\newtheorem{eq}[thm]{equation}
\newtheorem{que}{Problem}[section]
\newtheorem{ques}[thm]{Question}
\providecommand{\floor}[1]{\left\lfloor#1\right\rfloor}
\maketitle
\begin{abstract}
Union-free families of subsets of $[n]=\{1,\ldots n\}$ have been studied in \cite{FF}.  In this paper, we provide a complete characterization of maximal {\it symmetric difference}-free families of subsets of $[n]$.
\end{abstract}
\section{Introduction}
In combinatorics, there are many examples of set systems that obey certain properties, such as (i) Sperner families, i.e.~collections ${\cal A}$ of subsets of $[n]$ such that $A,B\in {\cal A}\Rightarrow A\subset B$ and $B\subset A$ are both false; and (ii) pairwise intersecting families of $k$-sets, i.e.~collections ${\cal A}$ of sets of size $k$ such that $A,B\in {\cal A}\Rightarrow A\cap B\neq\emptyset$.  In these two cases, Sperner's theorem and the Erd\H os-Ko-Rado theorem (see, e.g., \cite{vanlint}) prove that the corresponding extremal families are of size ${n\choose{\lfloor n/2\rfloor}}$ and ${{n-1}\choose{k-1}}$ respectively -- and correspond to the ``obvious" choices, namely {\it all} subsets of size ${\lfloor n/2\rfloor}$, and all subsets of size $k$ containing a fixed element $a$. Another example of set systems is that of   union closed families, i.e.~ensembles ${\cal A}$ such that $A,B\in {\cal A}\Rightarrow A\cup B\in{\cal A}$.  Frankl's conjecture states that in any { union closed} family ${\cal A}$ of subsets of $\{1,2,\ldots , n\}$ there exists an element $a$ that belongs to at least half the sets in the collection \cite{uiuc}.  {\it Union-free} families were defined by Frankl and F\"uredi \cite{FF} to be those families $\a$ where there do not exist $A,B,C,D\in\a$ such that $A \cup B = C \cup D $.  This definition is similar to that of $B_2$ (or {\it Sidon}) subsets $\a$ of integers \cite{obryant}, which are those for which each of the 2-sums $a+b, a\le b, a,b\in\a$ are distinct.

When defining delta-free families, we adopt an approach and a definition slightly different than those used in \cite{FF}; A family of sets ${\cal A}$ is delta free (which we will often write as $\Delta$-free) if for any $A,B\in{\cal A}$, $A \Delta B \not\in{\cal A}$.  Analogously with union closed families, however, we define $\Delta$-closed families $\a$ to be those for which $A,B\in{\a}\Rightarrow A\Delta B\in{\a}$.  (Recall that the symmetric difference $ A\Delta B$, of two sets $A,B$ contains all the elements that are in exactly one of the sets.)    To give two easy examples, the collection of $2^{n-1}$ subsets of $[n]$ that contain the element ``1", and any collection $\a $ of pairwise disjoint sets both form delta-free families, whereas the collection of even-sized subsets of $[n]$ forms a delta-closed system.  The union free families of \cite{FF} are those for which the ``knowledge base" of two people is never equal to the knowledge base of two other people.  Delta free/ Delta closed families, as defined above, are, on the other hand, those for which the ``special skill sets" of two people can never/always be replaced by the skills of a single third person.  Another ``application" of $ \Delta $-free families would be to let each set represent a person and each element represent a language that the person speaks.  If the family is $ \Delta $-free, then there is no person who can translate for two people all the languages those two do not share and yet not understand the languages those two other people do share.

In this paper we focus more on $\Delta$-free ensembles $\a$ than those that are $\Delta$-closed.

\section{Preliminary Results}

\begin{lem}
If a family ${\cal A}$ contains the empty set, ${\cal A}$ cannot be $\Delta $-free.
\end{lem}
\begin{proof}Assume ${\cal A}$ is $ \Delta $-free and contains the empty set.  Since $A \Delta \emptyset = A$ for any $A$ in ${\cal A}$ we reach a contradiction.  Equivalently, if ${\cal A}$ is $ \Delta $-free, then for any $A \in {\cal A}, A \Delta A = \emptyset $ so $\emptyset\not\in{\cal A}$. \end{proof}

\begin{lem}
Any family ${\cal A}$ that only contains sets of odd cardinality (which we will often call ``odd sets") is $ \Delta $-free.  Similarly, the collection $\a$ of all even sets is $\Delta$-closed, as are several smaller families of even sets.
\end{lem}
\begin{proof}Elementary.  Since
\[\vert A\Delta B\vert=\vert A\vert+\vert B\vert-2\vert A\cap B\vert,\]
it follows that the cardinality of the symmetric difference of two sets $A,B$ of the same parity is always even. \end{proof}
This lemma shows that, of the $2^{n}$ subsets of $\{1,\ldots n\}$, a collection consisting of exactly  half of these subsets is $\Delta$-free. Thus a maximal $ \Delta $-free family of subsets of $\{1,\ldots n\}$ has to have at least $2^{n-1}$ members.  In the spirit of the classical results of Sperner and Erd\H os-Ko-Rado which exemplify results that satisfy West's \cite{west} mnemonic ``TONCAS", or ``The Obvious Necessary Condition is Also Sufficient", we will show that   
\[{\cal A}\ {\rm is}\ \Delta-{\rm free}\Rightarrow\vert{\cal A}\vert\le2^{n-1}.\]
(West, in a personal communication, has attributed the acronym to Crispin St. John Alvah Nash-Williams.)
\section{An Upper Bound and the Class of Maximal Families}
A $ \Delta $-free family can easily be seen to contain sets with even cardinality.  For example, for $n=3$, the family $\{1\}, \{1,2\}, \{1,3\}, \{1,2,3\}$ contains sets with both even and odd cardinality. The family is $ \Delta $-free as it does not contain $\{2\}, \{3\},$ or $\{2,3\}$. It is worth noting that this family too has $2^{n-1}$  members. However, since half ($2^{n-1}$) of the subsets of $\{1,\ldots n\}$ are odd, in order to have more than $2^{n-1}$ members in a $\Delta$-free family, we would need to have at least one even set, which leads to our third lemma.
\begin{lem}
A maximal $ \Delta $-free family that contains at least one even set contains exactly half the even sets and exactly half the odd sets. 
\end{lem}
\begin{proof}Let ${\cal A}$ be a maximal $ \Delta $-free family on $\{1,\ldots n\}$ that contains at least one even set, say $F_{e}$. Since the family consisting of all even subsets cannot be delta-free, and since the cardinality of ${\cal A}$ is at least $2^{n-1}$, we let $O$ be any of the $2^{n-1}$ odd subsets of $[n]$ and consider the set $F_{e} \Delta O = O_{F_e}$, which is an odd set that is distinct from $O$. If $O \in \a,$ then $O_{F_e} \notin \a$ since $\a$ is $ \Delta $-free. Now $A\Delta B=A\Delta C\Rightarrow B=C$.  So, if $\a$ contains an even set, each odd set in $\a$ rules out another odd set as a potential member of $\a$. Therefore, if $\a$ contains an even set, it can contain at most half, or $2^{n-2}$, of the odd subsets of $[n]$. 

Similarly, let $E$ be any of the $2^{n-1}$ even subsets of $[n]$. If $E \in \a$, then we know that $E \neq \emptyset$. If $E = F_{e}$ then $E \Delta F_{e} = \emptyset \notin F$, so we may assume that $E$ is different from $F_{e}$.  $E \Delta F_{e}$ is even so each even set in $\a$ rules out another even set as a potential member of $\a$. So, again, if $\a$ contains an even set, it can contain at most half, or $2^{n-2}$, of the even sets, including the original even set.  Since a $ \Delta $-free family on $\{1,\ldots n\}$ that contains at least one even set can have at most $2^{n-2}$ of the odd sets and at most $2^{n-2}$ of the even sets, it can contain at most $2^{n-1}$ sets.  Since a maximal family has at least $2^{n-1}$ sets, the result follows. \end{proof}
We now know that a maximal $ \Delta $-free family consists of either all the odd sets or half the even sets and half the odd sets. In the latter case, we may ask ``which half of the even sets?" and ``which half of the odd sets?". 

An analysis of the $n=3$ case yielded the following $ \Delta $-free families:
$$\{\{1\}, \{2\}, \{3\}, \{1,2,3\}\},$$
which is the family with all the odd subsets;
$$\{\{1\}, \{1,2\}, \{1,3\}, \{1,2,3\}\};$$
two other families isomorphic to it; and three families isomorphic to
$$\{\{1\}, \{2\}, \{1,3\}, \{2,3\}\}.$$

For $n=4$, the non-isomorphic $ \Delta $-free families are as follows:

$$\{\{1\},\{2\},\{3\},\{4\},\{1,2,3\},\ \{1,2,4\},\{1,3,4\},\{2,3,4\}\},$$
$$\{\{1\},\{1,2\}, \{1,3\}, \{1,4\}, \{1,2,3\},\{1,2,4\}, \{1,3,4\},\{1,2,3,4\}\},$$
$$\{\{1\},\{2\},\{1,3\},\{1,4\},\{2,3\}, \{2,4\},\{1,3,4\},\  \{2,3,4\}\},$$
and 
$$\{\{1\},\{2\},\{3\},\{1,4\},\{2,4\},\{3,4\},\{1,2,3\},\{1,2,3,4\}\};$$
the above have zero, three, five, and three other isomorphic families respectively.  The $n=3$ and $n=4$ cases illustrate a general fact:  Let $S$ be the set of singletons that are to be in $\a$, so $S^{C}$ is the set of elements not represented by singletons in the family $\a$. We can construct the full ensemble $\a$ by adding to the singletons all of the odd sets with an even intersection with $S^{C}$ (which includes the singletons) and all of the even sets with an odd intersection with $S^{C}$. Consider the following example for $n=5$. If we know that $\{1\}$ and $\{2\}$ are in $\a$ but $\{3\}, \{4\},$ and $\{5\}$ are not, then $S^{C}$ is $\{3,4,5\}$.  Based on this, the one-element sets in $\a$ are $\{1\}$ and $\{2\}$ since they have an even, albeit empty, intersection with $S^{C}$.  The two-element sets are $\{1,3\}, \{1,4\},\{1,5\},\{2,3\},\{2,4\},$ and $\{2,5\}$, while the three-element sets are $\{1,3,4\},\{1,3,5\},\{1,4,5\},\{2,3,4\},\{2,4,5\}$, and $\{2,3,5\}$. The four-element sets are $\{1,3,4,5\}$ and $\{2,3,4,5\}$ and there are no five-element sets in $\a$ in this case. This particular family is $ \Delta $-free, but the question arises whether this method of construction works in general. Theorem 3.2 provides an affirmative answer.
\begin{thm}
Given a proper subset $S^{C}$ of $\{1,\ldots n\}$, we can construct a maximal $ \Delta $-free family $\a$ on $\{1,\ldots n\}$ by taking all the odd subsets of $[n]$ having an even intersection with $S^{C}$ and the even subsets of $[n]$ having an odd intersection with $S^{C}$.
\end{thm}
\begin{proof}Let $\a$ be a family on $\{1,\ldots n\}$ that consists of the odd members of ${\cal P}([n])$ whose intersection with a subset $S^{C}$ of $\{1,\ldots n\}$ is even, together with the even subsets of $[n]$ whose intersection with $S^{C}$ is odd.  Let $A, B\in\a$.\\ 
\textit{Case 1: $A$ and $B$ are both odd}\\
If $A$ and $B$ are both odd, then their symmetric difference would be even. $A$ and $B$ each have an even intersection with $S^{C}$, so the symmetric difference of their intersections with $S^{C}$ would be even, as well as a subset of $A \Delta B$. Thus, $A \Delta B$ would be an even member of the power set with an even intersection with $S^{C}$.  Hence $A \Delta B \notin \a$.\\ 
\textit{Case 2: $A$ and $B$ are both even}\\
If $A$ and $B$ are both even, then their symmetric difference would be even. $A$ and $B$ each have an odd intersection with $S^{C}$, so the symmetric difference of their intersections with $S^{C}$ would be even. Since $\a$ does not contain even subsets with an even intersection with $S^{C}$ we know that $A \Delta B \notin \a$.\\ 
\textit{Case 3: $A$ is even and $B$ is odd}\\
Since $A$ is even and in $\a$ it has an odd intersection with $S^{C}$. Since $B$ is odd and in $\a$ it has an even intersection with $S^{C}$. Also, $A \Delta B$ is odd. The symmetric difference of the intersection of $A$ with $S^{C}$ and the intersection of $B$ with $S^{C}$ is a subset of $A \Delta B$ and is odd. Therefore, $A \Delta B$ is an odd subset of $[n]$ with an odd intersection with $S^{C}$. It follows that $A \Delta B \notin \a$.
 
In any case, given $A, B\in {\a}, A \Delta B \notin \a$ so $\a$ is $ \Delta $-free.  Of the $2^{n-1}$ odd subsets of $[n]$, $2^{n-2}$ have an even intersection with $S^{C}$ and, of the $2^{n-1}$ even subsets, $2^{n-2}$ have an odd intersection with $S^{C}$. Combined, our family is of size $2^{n-2} + 2^{n-2} = 2^{n-1}$, so we have constructed a family that is $ \Delta $-free and is of maximal size. 
\end{proof}
Note that the family ${\a}={\a}(S^C)$ in Theorem 3.2 contains singletons sets, namely all one point subsets of $S$.  Note also that the complement of the family ${\a}(S^C)$ is $\Delta$-closed.

From the proof of Theorem 3.2, we know that our method will yield a $ \Delta $-free family of maximal size.  A family consisting of all the odd members of the power set can be constructed by letting $S^{C}$ be the empty set.  Since every odd set has a trivially even, empty intersection with $S^{C}$, they can all be included in $\a$.  Since every even set also has a trivially even, empty intersection with $S^{C}$, no such set can be included.  If we now let $S=\{1\}$ it is easy to see that we end up with the ``Erd\H os-Ko-Rado" $\Delta$-free family of all subsets of $[n]$ that contain the point 1.  Note that the family of all supersets of $\{1,2\}$ is also $\Delta$-free but not maximal.   Of course, we cannot let $S=\emptyset$ in Theorem 3.2 since all the odd/even subsets would have an odd/even intersection with $S^C$, and thus we would have ${\a}=\emptyset$, which is not a maximal family.  Thus our process yields $2^n-1$ maximal delta-free families, as evidenced above for $n=3,4$.  For a given $S^C$ we will denote the associated $\Delta$-free family by ${\a}(S^C)$, and call it the $\Delta$-free family {\it generated} by the set $S^C$.

We know the method outlined in Theorem 3.2 will construct maximal $ \Delta $-free families, but in order to completely characterize these, we need to make sure that the ensembles ${\a}(S^C)$ uncovered by Theorem 3.2 are the only maximal $ \Delta $-free families.  Let us first explore a cognate issue.  Given any collection ${\a}=\{A_1,A_2,\ldots,A_N\}$ of $\Delta$-free subsets of $[n]$ and any set $T$, we can partition the sets in $\a$ into four categories, which we can call ${\a}_{o,e}, {\a}_{e,o}, {\a}_{o,o}, {\a}_{e,e}$, where, e.g., ${\a}_{o,e}$ consist of those elements of $\a$ that are odd sets having an even intersection with $T$.  For specificity, note that the $\Delta$-free family of subsets of $[4]$ given by 
$$\{\{1\},\{2\},\{1,3\},\{1,4\},\{2,3\}, \{2,4\},\{1,3,4\},\  \{2,3,4\}\}$$
{\it does} correspond to a maximal family ${\a}(S^C)$ with $S^C=\{3\}$, but it {\it also} could be viewed as a collection of sets with specific intersection patterns with {\it another} set such as $T=\{1,2,3\}$ -- in which case we would have
$${\a}_{o,o}=\{\{1\},\{2\}\},$$
$${\a}_{e,o}=\{\{1,4\},\{2,4\}\},$$
$${\a}_{o,e}=\{\{1,3,4\},\{2,3,4\}\},$$
and
$${\a}_{e,e}=\{\{1,3\},\{2,3\}\}.$$
In fact, we next show that in any maximal $\Delta$-free family and for any set $T\ne S^C$, $$\vert {\a}_{o,e}\vert=\vert {\a}_{e,o}\vert=\vert {\a}_{o,o}\vert=\vert {\a}_{e,e}\vert=2^{n-3}.$$
Assume that the family has at least one even set, and thus  that there are $2^{n-2}$ even and $2^{n-2}$ odd sets.  Let
$$\vert {\a}_{e,o}\vert=2^{n-2}-x:=u; \vert {\a}_{e,e}\vert=x$$
$$\vert {\a}_{o,e}\vert=2^{n-2}-y:=v; \vert {\a}_{o,o}\vert=y,$$
where at least one of $x,y$ is positive, and at least one of $u,v,x,y$ is larger than $2^{n-3}$.  The following diagram represents where the symmetric difference of two sets, each in one of the classes ${\a}_{o,e}, {\a}_{e,o}, {\a}_{o,o}, {\a}_{e,e}$, lies:

$$\begin{matrix}\ &{\a}_{o,o}&{\a}_{o,e}&{\a}_{e,o}&{\a}_{e,e}\cr
{\a}_{o,o}&{\a}_{e,e}&{\a}_{e,o}&{\a}_{o,e}&{\a}_{o,o}\cr
{\a}_{o,e}&{\a}_{e,o}&{\a}_{e,e}&{\a}_{o,o}&{\a}_{o,e}\cr
{\a}_{e,o}&{\a}_{o,e}&{\a}_{o,o}&{\a}_{e,e}&{\a}_{e,o}\cr
{\a}_{e,e}&{\a}_{o,o}&{\a}_{o,e}&{\a}_{e,o}&{\a}_{e,e}\end{matrix}
$$
For example, if $A\in{\a}_{o,e}$ and $B\in{\a}_{e,o},$ then $A\Delta B\in{\a}_{o,o}$.
So, if, without loss of generality, $x\ge 1 $ and $u>2^{n-3}$, there is a surfeit of sets in the class ${\a}_{e,o}$.  But each set in ${\a}_{e,e}$, eliminates one of the sets in ${\a}_{e,o}$, when it is ``symmetrically differenced" with a set in class ${\a}_{e,o}$.  Hence it is impossible for there to be more than $2^{n-3}$ sets in the class ${\a}_{e,o}$.  This proves the claim.  We are now ready to state our main result.
\begin{thm}
The procedure outlined in Theorem 3.2 generates all the possible maximal $ \Delta $-free families on $\{1,\ldots n\}$.  In other words, each $\Delta$-free $\a$ equals ${\a}(S^C)$ for some $S^C$ (in which, by definition, ${\a}_{o,o}={\a}_{e,e}=\emptyset$).
\end{thm}
\begin{proof}
We start by showing that every maximal family must contain singleton subsets of $[n]$.  Let's assume that a maximal family contains no singletons; since it doesn't contain only the odds, it must contain half the evens.  If it contains a 2-element set, say $\{1,2\}$, we would normally have to choose between $\{1\}$ and $\{2\}$ in constructing a maximal set, just based on how evens exclude half the odds.  We can choose neither, though, since both are singletons -- leaving us with a maximum of $2^{n-1}  -1$ total possibilities, so if we don't have any singletons, we cannot have any two-element sets in a maximal family.  The next possibility for an even set is a 4-element set, but if we have $\{1,2,3,4\}$ in the family, we'd normally have to choose between, say, $\{1,2\}$ and $\{3,4\}$.  Since, however, we just showed that there cannot be any any two element sets in a maximal family if we don't have singletons, neither of these sets can be in the family.  Likewise, any 6-element set would normally have us picking between a 2- and 4-element set to make a maximal set, neither of which we can have in a maximal set.  The claim follows by induction.

Let $S^C$ denote the set $[n]\setminus\{j:\{j\}\in{\a}\}$.  

\noindent \textit{Case 1: $\a$ contains an even set with an even intersection with $S^{C}$}. 
For a two element set to have an even intersection with $S^{C}$ it must contain either no elements or two elements from $S^{C}$. Let the proposed two element set be $\{1,2\}$. If $\{1,2\} \cap S^{C} = \emptyset$ then $\{1\} \in S$ so, by definition, $\{1\} \in {\a}$.  Similarly, $\{2\} \in {\a}$. However, $\{1\} \Delta \{2\} = \{1,2\}$, so if $\{1,2\} \cap S^{C} = \emptyset$, then $\{1,2\} \notin {\a}$, a contradiction to our assumption that $\{1,2\} \in {\a}$. If $\{1,2\}$ contains exactly two elements from $S^{C}$, then $\{1\} \notin {\a}$ and $\{2\} \notin {\a}$ since $\{1,2\}\subset S^{C}$.  However, in order to be maximal, if $\{1,2\} \in {\a}$, then exactly one of $\{1\}$ and $\{2\}$ must be in ${\a}$ as well, which is a contradiction. So, if $\a$ contains a two element set with an even intersection with the corresponding $S^{C}$, it cannot be maximal.

We next verify the details for four element sets:  If $\{1,2,3,4\}$ has an even intersection with $S^{C}$, then there are three possibilities, which, without loss of generality, are (i) $\{1\},\{2\},\{3\},\{4\}\in{\a}$; (ii) $\{1\},\{2\}\in{\a},\{3,4\}\subset S^C$; and (iii) $\{1,2,3,4\}\subset S^C$.  Now if $\{1,2,3,4\}$ belongs to $\a$ then exactly one of $\{1,2\}$ and $\{3,4\}$ must belong to $\a$ too.  But under each of the above three scenarios, both $\{1,2\}$ and $\{3,4\}$ are even two-element sets with an even intersection with $S^C$, and so $\a$ cannot be maximal.

From here, we can use induction.  We know we cannot have a two-element set with an even intersection with $S^{C}$ in a maximal $ \Delta $-free family. For $k\ge 2$, suppose we cannot have an even set of size $2k$ or less, having an even intersection with $S^{C}$, in a maximal $ \Delta $-free family. A $2(k+1)$-element set with an even intersection with $S^{C}$may split up in two ways:  Either we have all of its elements $\{a_1,\ldots,a_{2k+2}\}$ in $\a$, or $\{a_1,\ldots,a_{2k+2}\}\subset S^C$, or there are two non-empty even sets $A, B$, $\vert A\cup B\vert=2k+2$ such that the elements of $A$ are in $\a$, and $B\subset S^C$.  In the first or second case, we see that either $\{a_1,a_2\}$ or $\{a_3,\ldots,a_{2k+2}\}$ must be in $\a$, and in the second case, either $A$ or $B$ must be in $\a$.  Under any scenario, we contradict maximality, since each of the partitioned sets in question are even, of size $\le 2k$,  and have an even intersection with $S^C$.

\noindent \textit{Case 2: $\a$ contains an odd set with an odd intersection with $S^{C}$.}
We now know that in a maximal $ \Delta $-free family the only even sets possible are those with an odd intersection with $S^{C}$ and that there are $2^{n-2}$ such sets. The symmetric difference of an even set with an odd intersection with $S^{C}$ and an odd set with an odd intersection with $S^{C}$  is an odd set with an even intersection with $S^{C}$. Therefore, if the family under consideration contains all the even sets with an odd intersection with $S^{C}$ and an odd set with an odd intersection with $S^{C}$, it cannot contain an odd set with an even intersection with $S^{C}$ including the singletons we used to define $S$ (and thus $S^{C}$), which leads to a contradiction.  Therefore, if a family contains an odd set with an odd intersection with the corresponding $S^{C}$, it cannot be a maximal $ \Delta $-free family. 
\end{proof}

\section{Open Questions}  How would our main results have changed had we used the ``traditional" definition of $\Delta$-free families $\a$ as those for which there do not exist sets $A,B,C,D\in\a$ with $A\Delta B=C\Delta D$?  If we draw sets from ${\cal P}([n])$ independently and at random with probability $p$, what is the threshold probability $p_0$ for the disappearance  of the $\Delta$-free property?  What closure or avoidance properties can be generalized to three or more sets, where the relevant conditions might involve arithmetic modulo three?
\section{Acknowledgments} The research of AG was supported by NSF Grants 0552730 and 1004624.

\end{document}